\documentclass[12pt,a4paper]{article}
\usepackage{amssymb,amsmath,amsthm}
\usepackage{graphicx}
\usepackage{hyperref}

\newcommand\cF{{\mathcal F}}

\newcommand\cG{{\mathcal G}}
\newcommand\cH{{\mathcal H}}

\theoremstyle{plain}
\newtheorem{theorem}{Theorem}[section]

\newtheorem{proposition}[theorem]{Proposition}

\theoremstyle{definition}

\newtheorem{claim}[theorem]{Claim}

\newcommand\cref[1]{Corollary~\ref{cor:#1}}

\title{On saturation of Berge hypergraphs}
\author{D\'aniel Gerbner$^1$ \ \ Bal\'azs Patk\'os$^{1,2}$ \ \ Zsolt Tuza$^{1,3}$ \ \ M\'at\'e Vizer$^1$ \\ \small $^1$ Alfr\'ed R\'enyi Institute of Mathematics\\ \small $^2$ Moscow Institute of Physics and Technology\\ \small $^3$ Department of Computer Science and Systems
 Technology, University of Pannonia%, Hungarian Academy of Sciences
}
\date{}
\begin{document}

\maketitle

\begin{abstract}
    A hypergraph $H=(V(H), E(H))$ is a \emph{Berge copy of a graph $F$}, if $V(F)\subset V(H)$ and there is a bijection $f:E(F)\rightarrow E(H)$ such that for any $e\in E(F)$ we have $e\subset f(e)$. A hypergraph is Berge-$F$-free if it does not contain any Berge copies of $F$. We address the saturation problem concerning Berge-$F$-free hypergraphs, i.e., what is the minimum number $sat_r(n,F)$ of hyperedges in an $r$-uniform Berge-$F$-free hypergraph $H$ with the property that adding any new hyperedge to $H$ creates a Berge copy of $F$. We prove that $sat_r(n,F)$ grows linearly in $n$ if $F$ is either complete multipartite or it possesses the following property: if $d_1\le d_2\le \dots \le d_{|V(F)|}$ is the degree sequence of $F$, then $F$ contains two adjacent vertices $u,v$ with $d_F(u)=d_1$, $d_F(v)=d_2$. In particular, the Berge-saturation number of regular graphs grows linearly in $n$. 
\end{abstract}
\section{Introduction}

Given a family $\cF$ of (hyper)graphs, we say that a (hyper)graph $G$ is $\cF$-free if $G$ does not contain any member of $\cF$ as a subhypergraph. The obvious question is how large an $\cF$-free (hyper)graph can be, i.e.\ what is the maximum number $ex(n,\cF)$ of (hyper)edges in an $\cF$-free $n$-vertex (hyper)graph is called the extremal/Tur\'an problem. A natural counterpart to this well-studied problem is the so-called \textit{saturation problem}. We say that $G$ is \textit{$\cF$-saturated} if $G$ is $\cF$-free, but adding any (hyper)edge to $G$ creates a member of $\cF$. The question is how small an $\cF$-saturated (hyper)graph can be, i.e.\ what is the minimum number $sat(n,\cF)$ of (hyper)edges in an $\cF$-saturated $n$-vertex (hyper)graph. 

In the graph case, the study of saturation number was initiated by Erd\H os, Hajnal, and Moon \cite{ehm}.
Their theorem on complete graphs was generalized to complete uniform hypergraphs by Bollob\'as \cite{B65}.
K\'aszonyi and Tuza \cite{kt} showed that for any family $\cF$ of graphs, we have $sat(n,\cF)=O(n)$. For hypergraphs, Pikhurko \cite{P2004} proved the analogous result that for any family $\cF$ of $r$-uniform hypergraphs, he proved that we have $sat(n,\cF)=O(n^{r-1})$.
For some further types of saturation (``strongly $F$-saturated'' and ``weakly $F$-saturated'' hypergraphs) the exact exponent of $n$ was determined in \cite{t92} for every forbidden hypergraph $F$.

In this paper, we consider some special families of hypergraphs.
We say that a hypergraph $H$ is a \emph{Berge copy of a graph $F$} (in short: $H$ is a Berge-$F$) if $V(F)\subset V(H)$ and there is a bijection $f:E(F)\rightarrow E(H)$ such that for any $e\in E(F)$ we have $e\subset f(e)$. We say that $F$ is a \textit{core graph} of $H$. Note that there might be multiple core graphs of a Berge-$F$ hypergraph and $F$ might be the core graph of multiple Berge-$F$ hypergraphs.

Berge hypergraphs were introduced by Gerbner and Palmer \cite{gp1}, extending the notion of hypergraph cycles in Berge's definition \cite{Be87}. They studied the largest number of hyperedges in Berge-$F$-free hypergraphs (and also the largest total size, i.e.\ the sum of the sizes of the hyperedges). English, Graber, Kirkpatrick, Methuku and Sullivan \cite{EGKMS2017} considered the saturation problem for Berge hypergraphs. They conjectured that  $sat_r(n,\text{Berge-}F)=O(n)$ holds for any $r$ and $F$, and proved it for several classes of graphs. Here and throughout the paper the parameter $r$ in the index denotes that we consider $r$-uniform hypergraphs, and we will denote $sat_r(n,\text{Berge-}F)$ by $sat_r(n,F)$ for brevity. The conjecture was proved for $3\le r\le 5$ and any $F$ in \cite{EGMT2019}. In this paper we gather some further results that support the conjecture.

English, Gerbner, Methuku and Tait \cite{EGMT2019} extended this conjecture to hypergraph-based Berge hypergraphs. Analogously to the graph-based case, we say that a hypergraph $H$ is a Berge copy of a hypergraph $F$ (in short: $H$ is a Berge-$F$) if $V(F)\subset V(H)$ and there is a bijection $f:E(F)\rightarrow E(H)$ such that for any $e\in E(F)$ we have $e\subset f(e)$. We say that $F$ is a core hypergraph of $H$. The conjecture in this case states that if $F$ is a $u$-uniform hypergraph, then $sat_r(n,F)=O(n^{u-1})$.
%Note that there might be multiple core graphs of a Berge-$F$ and $F$ might be the core graph of multiple Berge-$F$ hypergraphs.

For a hypergraph $H=(V(H),E(H))$ and a family of hypergraphs $\cF$ we say that $H$ is \emph{$\cF$-oversaturated} if for any hyperedge $h \subset V(H)$ that is not in $H$, there is a copy of a hypergraph $F\in\cF$ that consists of $h$ and $|E(F)|-1$ hyperedges in $E(H)$.
Let $osat_r(n,\cF)$ denote the smallest number of hyperedges in an $\cF$-oversaturated $r$-uniform hypergraph on $n$ vertices.

\begin{proposition}\label{oversa} For any\/ $u$-uniform hypergraph\/ $F$ and any\/ $r>u$, we have\/
$osat_r(n,F)=O(n^{u-1})$. Moreover, there is an\/ $r$-uniform hypergraph\/ $H$ with\/ $O(n^{u-1})$ hyperedges such that adding any hyperedge to\/ $H$ creates a Berge-$F$ such that its core hypergraph\/ $F_0$ (which is a copy of\/ $F$) is not a core hypergraph of any Berge-$F$ in\/ $H$.
\end{proposition}

We remark that in the case $u=2$, the linearity of $osat_r(n,\cF)$ follows from either of the next two theorems, as they imply $sat_r(n,K_k)=O(n)$. Indeed, if $v$ is the number of vertices of any graph in $\cF$, then any Berge-$K_v$-saturated hypergraph is obviously Berge-$\cF$-oversaturated.

\begin{theorem}\label{complmulti}
For any $r,s \ge 2$ and any sequence of integers\/ $1 \le k_1 \le k_2 \le \dots \le k_{s+1}$ we have\/ $$sat_r(n,K_{k_1,k_2,\dots,k_{s+1}})=O(n).$$
\end{theorem}

Let $F$ be a fixed graph on $v$ vertices with degree sequence $d_1\le d_2\le \dots \le d_v$. Set $\delta:=d_2-1$. We say that $F$ is of \textit{type I} if there exist vertices $u_1$, $u_2$ with $d_F(u_1)=d_1$, $d_F(u_2)=d_2$ that are joined with an edge. Otherwise $F$ is called of type II. Observe that any regular graph is of type I. 

\begin{theorem}\label{type1}
For any graph\/ $F$ of type I and any\/ $r\ge 3$ we have\/ $sat_r(n,F)=O(n)$. 
\end{theorem}

\section{Proofs}

\begin{proof}[Proof of Proposition \ref{oversa}]
We define $V(H)$ as the disjoint union of a set $R$ of size $r-u$ and a set $L$ of size $n-r+u$. We take an $F$-saturated $u$-uniform hypergraph $G$ with vertex set $L$ that contains $O(n^{u-1})$ 
%cliques
hyperedges. Such a hypergraph exists by the celebrated result of Pikhurko \cite{P2004}.
As another alternative, one may take an oversaturated hypergraph with $O(n^{u-1})$ hyperedges, whose existence is guaranteed by \cite[Theorem 1]{t92}.
Then we let the hyperedges $h$ of $H$ be the $r$-sets with the property that $h\cap L$ is a hyperedge of $G$ 
%where the empty set is considerd a clique. 
or has at most $u-1$ vertices from $L$.

Obviously $H$ has $O(n^{u-1})$ hyeredges. Clearly we have that every $r$-set $h$ that is not a hyperedge of $H$ contains a $u$-element subset $e$ of $L$ that is not a hyperedge of $G$.
%(otherwise $h$ would intersect $L$ in a clique of $G$)
Then $e$ creates a copy of $F$. We let $f(e)=h$, and for each other edge $e'$ of that copy of $F$, we let $f(e')=e'\cup R$. This shows that this copy of $F$ is the core of a Berge-$F$.
\end{proof}

%\begin{theorem}\label{oversa} For any graph $F$ and any $r\ge 2$, we have $osat_r(n,\text{Berge-}F)=O(n)$. Moreover, there is an $r$-graph $H$ with $O(n)$ hyperedges such that adding any hyperedge to $H$ creates a Berge-$F$ such that its core graph $F_0$ (which is a copy of $F$) is not a core graph of any Berge-$F$ in $H$.
%\end{theorem}

%\begin{proof}[Proof of Theorem \ref{oversa}]
%First we take a set $R$ of size $r-2$ and a set $L$ of size $n-r+2$. We take an $F$-saturated graph $G$ with vertex set $L$ that contains $O(n)$ 
%cliques
%edges. Then we let the hyperedges $h$ of $H$ be the $r$-sets with the property that $h\cap L$ is an edge 
%where the empty set is considerd a clique. 
%or at most one vertex from $L$.

%Obviously $H$ has $O(n)$ hyeredges. Clearly we have that every $r$-set $h$ that is not a hyperedge of $H$ contains a sub-edge $e$ in $L$ that is not in $G$
%(otherwise $h$ would intersect $L$ in a clique of $G$)
%. Then $e$ creates a copy of $F$. We let $f(e)=h$, and for each other edge $e'$ of that copy of $F$, we let $f(e')=e'\cup R$. This shows that this copy of $F$ is the core of a Berge-$F$.
%\end{proof}

\begin{proof}[Proof of Theorem \ref{complmulti}]
We consider two cases according to how large the uniformity $r$ is compared to the sum of class sizes $k_1,k_2,\dots,k_{s+1}$. We set $N:=\sum_{i=1}^sk_i-1$. For brevity, we write $K$ for $K_{k_1,k_2,\dots,k_{s+1}}$.

\vskip 0.15truecm

\textsc{Case I. } $r\le \sum_{i=1}^{s+1}k_i-3$

\vskip 0.1truecm

Let $(C,B_1,B_2,\dots,B_m,R)$ be a partition of $[n]$ with $|C|=N$, $|B_i|=k_{s+1}$ for all $i \le m$ where $m=\lfloor \frac{n-N}{k_{s+1}} \rfloor$, and $|R|\equiv n-N ~(mod ~k_{s+1})$, $|R|<k_{s+1}$. Consider the family $\cG=\{A\in\binom{[n]}{r}: A\subset C \cup B_i ~\text{for some $i$}\}$. Observe that $\cG$ is Berge-$K$-free. Indeed, a copy of a Berge-$K$ must contain a vertex $v$ in the smallest $s$ classes of the core from outside $C$. But then, if $v\in B_i$, either the whole copy is in $C\cup B_i$ or $C$ must contain all classes of the core of the copy. As none of these are possible, $\cG$ is indeed Berge-$K$-free.

Next observe that adding any $r$-set $G$ to $\cG$ that contains two vertices $u$ and $v$ from different $B_i$s, say $u\in B_i,v\in B_j$ $(i \neq j)$, would create a copy of a Berge-$K$. Indeed, the assumption $r\le \sum_{i=1}^{s+1}k_i-3$ ensures that there exist bijections $f_i:\binom{C\cup B_i}{2} \rightarrow \binom{C\cup B_i}{r}$ with $e\subset f(e)$. Then vertices of $C$ and $u$ can play the role of the $s$ smallest classes of $K$, and $\{v\}\cup B_i\setminus \{u\}$ can play the role of the largest class of $K$.

This shows that the additional hyperedges of any $K$-saturated family that contains $\cG$ are subsets of $C\cup B_i\cup R$ for some $i$, and hence there is only a linear number of them. As $\cG$ also contains a linear number of hyperedges, the total size of such $K$-saturated families is $O(n)$. 

\vskip 0.15truecm

\textsc{Case II. } $r\ge \sum_{i=1}^{s+1}k_i-2$

\vskip 0.1truecm

Once again, we define a partition $(C,B_1,B_2,\dots,B_m,R)$ of $[n]$ with $|C|=N$, but now with $|B_i|=r-2$ and $|R|\equiv n-N ~(mod ~r-2)$, $|R|<r-2$. Let $x_1,x_2,\dots,x_{N}$ be the elements of $C$, let $e_1,e_2,\dots,e_{\binom{N}{2}}$ be the edges of the complete graph on $C$, and finally let $\pi_1, \pi_2,\dots,\pi_{\binom{N}{2}}$ be permutations of $C$ such that the endvertices of $e_i$ are the values $\pi_i(j),\pi_i(j+1)$ for some~$j$.

Then let us define the family $\cG$ as
\[
\cG=\bigcup_{i=1}^m\{\{x_{\pi_h(j)},x_{\pi_h(j+1)}\}\cup B_i:1\le j\le N, \ h\equiv i \ (\mbox{\rm mod } {\textstyle \binom{N}{2})\}}.
\]

First, we claim that $\cG$ is Berge-$K$-free. Indeed, there are only $N$ vertices with degree at least $N+1$.

Next, observe that if we add a family $\cF$ to $\cG$ that contains a Berge-$S_{k_{s+1}(r-2)}$ (a star with $k_{s+1}(r-2)$ leaves) with core completely in $\cup_{i=1}^m B_i$, then $\cG\cup \cF$ contains a Berge-$K$. Indeed, if $v$ is the center of the star, then $C\cup \{v\}$ plays the role of the smallest classes of $K$, and $k_{s+1}$ leaves that belong to distinct $B_i$s can play the role of the largest class of $K$. Here, we use the facts that every edge in $\binom{C}{2}$ is contained in an unbounded number of hyperedges of $\cG$ as $n$ tends to infinity and that for any vertex $u\in\cup_{i=1}^mB_i$, $\cG$ contains a Berge-star with center $u$ and core $C\cup \{u\}$; and if $u$ and $u'$ belong to different $B_i$'s, then the hyperedges of these Berge-stars are distinct.

Let $\cF$ be such that $\cF\cup \cG$ is Berge-$K$-free. Then by the above, $\cF'=\{F\cap (\cup_{i=1}^mB_i):F\in \cF\}$ is Berge-$S_{k_{s+1}(r-2)}$-free. Note that $|\cF|\le 2^{|C\cup R|}|\cF'|$, thus showing that $|\cF'|=O(n)$ finishes the proof. It is well-known that forbidding a Berge-star (or any Berge-tree) results in $O(n)$ hyperedges, but for sake of completeness we include a proof for stars. 

Observe that $\cF'$ being Berge-$S_{k_{s+1}(r-2)}$-free is equivalent to the condition that for every $x\in \cup_{i=1}^mB_i$ the family $\{F\setminus \{x\}:x\in F\in \cF'\}$ is not disjointly $k_{s+1}(r-2)$-representable, i.e.\ there do not exist $y_1,y_2,\dots,y_{k_{s+1}(r-2)}$ and sets $F_1,F_2,\dots,F_{k_{s+1}(r-2)}\in \cF'_x$ with $y_\alpha \in F_\alpha\setminus \cup_{j=1,j\neq \alpha}^{k_{s+1}(r-2)}F_j$. By a well-known result of Frankl and Pach \cite{FP} if all sets 
of a family $\cH$ with this property have size at most $r$, then $|\cH|$ is bounded by a constant depending only on $r$ and $k_{s+1}(r-2)$. That is, $|\cF'_x|$ is bounded by the same constant independently of $x$, and therefore the size of $\cF'$, and thus the size of $\cG$ is linear. We obtained that any $k$-saturated family $\cG'$ with $\cG\subset \cG'$ has $O(n)$ hyperedges.
\end{proof}

\begin{proposition}\label{isoedge}
Let\/ $F$ be a graph with no isolated vertex and with an isolated edge\/ $(u_1,u_2)$.
Then for any\/ $r\ge 3$ we have\/ $sat_r(n,F)=O(n)$.
\end{proposition}

\begin{proof}
Let $U$ be a set of size $n$. Let $v$ denote the number of vertices of $F$, let $F'$ be the graph obtained from $F$ by removing the edge $(u_1,u_2)$ and let $C$ be a $(v-2)$-subset of $U$. Suppose first $r\le v-1$, and let $\cG_0$ be a Berge copy of $F'$ with core $C$ and $\cG_0 \subseteq \cG_{C,1}\subseteq \binom{U}{r}$, where $\cG_{C,1}$ is the set of $r$-sets that contain at most one vertex from $U \setminus C$. Note that $\cG_{C,1}$ contains a linear number of $r$-subsets. Then let $\cG$ be an $r$-graph with $\cG_0\subseteq \cG \subseteq \cG_{C,1}$ such that any $H \in \cG_{C,1}\setminus \cG$ creates a Berge copy of $F$ with $\cG$. Then $\cG$ has linearly many hyperedges and is clearly $F$-saturated since if $G$ contains at least two vertices from $U \setminus C$, then $G$ can play the role of $(u_1,u_2)$ and together with the Berge copy of $F'$ they form a Berge-$F$. 

If $r\ge v$, then any $\cG$ with $e(F)-1$ $r$-subsets sharing $v-2$ common elements (denote their set by $C$) is $F$-saturated. Indeed, any additional $r$-set $G$ contains at least 2 vertices not in $C$, so those two vertices can play the role of $u_1$ and $u_2$, $G$ can play the role of the edge $(u_1,u_2)$, and the $r$-sets of $\cG$ form a Berge copy of $F'$ with core $C$. 
\end{proof}

Observe that if $F$ is of type I, then it cannot contain isolated vertices, and since graphs with an isolated edge
are covered by Proposition \ref{isoedge}, we may and will assume that $d_2-1=\delta\ge 1$
holds.

\begin{proof}[Proof of Theorem \ref{type1}]
Let $F$ be a graph of type I on $v$ vertices and let $u_1,u_2$ be a pair of vertices of $F$ showing the type I property. Set $d:=|N(u_1)\cap N(u_2)|$ and let $F'$ denote the subgraph of $F$ on $N(u_1)\cap N(u_2)$ spanned by the edges incident to $u_1$ or $u_2$ with the edge $(u_1,u_2)$ removed. Our strategy to prove the theorem is to construct a Berge-$F$-free $r$-graph $\cG$ with $O(n)$ hyperedges 
such that any $F$-saturated $r$-graph $\cG'\supset \cG$ contains at most a linear number of extra hyperedges.

Let us say that $\cG$ is $F$-good if its vertex set $V$ can be partitioned into $V=C\cup B_1\cup B_2\cup \dots \cup B_m \cup R$ such that $|C|=v-2$, all $B_i$'s have equal size $b$ at most $r$, $|R|<b$ and the following properties hold:
\begin{enumerate}
    \item 
    every hyperedge of $\cG$ not contained in $C$ is of the form $A \cup B_i$  for some $i=1,2,\dots,m$ with $A \subset C$,
    \item
    every vertex $u\in \cup_{i=1}^mB_i$ has degree $\delta$ in $\cG$,
    \item
    for every $1\le i<j\le m$ and $y\in B_i, y'\in B_j$, the sub-$r$-graph $\{G\in \cG:y\in G\}\cup \{G\in \cG:y'\in G\}$ contains a Berge-$F'$ with $y,y'$ being the only vertices of the core not in $C$ and $y,y'$ playing the role of $u_1,u_2$,
    \item
    for any $1\le i<j\le m$ there exist $\binom{v-2}{2}$ hyperedges $G_1,G_2,\dots, G_{\binom{v-2}{2}}\in\cG$ that are disjoint from $B_i\cup B_j$ and if $e_1,e_2,\dots,e_{\binom{v-2}{2}}$ is an enumeration of the edges of the complete graph on $C$, then $e_h\subset G_h$ for all $h=1,2,\dots,{\binom{v-2}{2}}$, i.e., these hyperedges form a Berge-$K_{v-2}$ with core $C$.
\end{enumerate}

\begin{claim}\label{extend}
If $\cG$ is $F$-good, then $\cG$ is Berge-$F$-free and any $F$-saturated supergraph $\cG'$ of $\cG$ contains at most a linear number of extra edges compared to $\cG$.
\end{claim}

\begin{proof}[Proof of Claim.]
Observe first that $\cG$ is Berge-$F$-free as the core of a copy of a Berge-$F$ should contain at least two vertices not in $C$, both of degree $\delta<d_2$.

Next, we claim that for any hyperedge $H$ meeting two distinct $B$'s, say $B_i$ and $B_j$, the $r$-graph $\cG\cup \{H\}$ contains a Berge-$F$. Indeed, let $y\in B_i\cap H, y'\in B_j\cap H$. Then by item 3 of the $F$-good property, $y$ can play the role of $u_1$ and $y'$ can play the role of $u_2$, $H$ can play the role of the edge $(u_1u_2)$, and item 4 of the $F$-good property ensures that the other vertices of $C$ can play the role of the rest of the core of $F$.

Finally, let $\cG'$ be any $F$-saturated $r$-graph containing $\cG$. Then by the above, any hyperedge in $\cG'\setminus \cG$ meets at most one $B_i$, and thus is of the form $P\cup Q$ with $P\subset C\cup R$, $Q\subset B_i$ for some $i$. The number of such sets is at most $2^b\, 2^{v-2+b}\, m=O(n)$. 
\end{proof}

\begin{claim}\label{Fgood}
For any type I graph $F$ on $v\ge 7$ vertices with $\delta>0$ and any integer $r\ge 6$ there exists an $F$-good $r$-graph $\cG$ with $O(n)$ hyperedges.
\end{claim}
\begin{proof}[Proof of Claim.]
%We consider several cases according to $v,d$ and $r$. In all cases w
We fix a set $D\subset C$ of size $d$.

\vskip 0.15truecm

\textsc{Case I. } $r\le v-4$

Then putting all $r$-subsets of $C$ into $\cG$ ensures item 4 of the $F$-good property. We set $b=r-2$, so all further sets will meet $C$ in $2$ vertices. Observe that $d_1-1-d+\delta-d\le v-2-d$ is equivalent to $2(d_1-1-d)+\delta-d_1+1\le v-2-d$. Let $D=\{x_1,x_2,\dots,x_d\}$, and $y_{1,1},y_{1,2},y_{2,1},y_{2,2},\dots,y_{d_1-1-d,1},y_{d_1-1-d,2}$, $z_1,z_2,\dots,z_{\delta-d_1+1}$ be distinct vertices in $C$. Note that $d_1-1\le \delta$ implies that $\delta-d_1+1$ is non-negative. Then let 
\[
\cG_0:=\{\{x_\ell,x_{\ell+1}\}:1\le \ell \le d\} \cup \{\{y_{\ell,1},y_{\ell,2}\}:1\le \ell \le d_1-1-d\} \cup \{\{z_\ell,z_{\ell+1}\}:1\le \ell \le \delta-d_1+1\},
\]
where addition is always modulo the underlying set, i.e., $\cG_0$ consists of two cycles and a matching.
Let us put all sets of the form $A\cup B_h$ with $A\in \cG_0$ and $1\le h \le m$ into $\cG$. Then items 1 and 2 of the $F$-good property are satisfied, thus we need to check item 3.

For any $1\le i<j\le m$, $D$ plays the role of $N(u_1)\cap N(u_2)$ and the hyperedges $B_i\cup \{x_\ell,x_{\ell+1}\}$, $B_j\cup \{x_\ell,x_{\ell+1}\}$
play the role of the edges connecting $u_1$, $u_2$ to vertices of $D$, respectively. Vertices $y_{\ell,1}$ play the role of vertices in $N(u_1)\setminus (N(u_2)\cup \{u_2\})$, while vertices $y_{\ell,2}$ with $\ell=1,2,\dots, d_1-1-d$ and $z_{\ell'}$ with $\ell'=1,2,\dots,\delta-d_1+1$ play the role of $N(u_2)\setminus (N(u_1)\cup \{u_1\})$. The use of hyperedges as edges is straightforward.

\vskip 0.15truecm

\textsc{Case II. } $r> v-4$

Then we set $b=r-(v-4)$ and thus every hyperedge meets $C$ in $c:=v-4=|C|-2$ vertices. Consequently, $|R|$ is the residue of $n-v+2$ modulo $b$. 
By $v\ge 7$, we obtain $c\ge 3$. Let $e_1,e_2,\dots,e_{\binom{v-2}{2}}$ be an enumeration of the edges of the complete graph on $C$. Then for any $1\le h\le m$, we will put a hyperedge of the form $A_{1,h}\cup B_h$ with $e_\alpha\subset A_{1,h}\subset C$ where $\alpha\equiv h ~(mod \binom{v-2}{2})$. As $n$ tends to infinity, so does $m$, and this will ensure item 4 of the $F$-good property. 

Suppose first $d>0$ and observe that $(d_1-1)+\delta\le v-2-d$, as $d_1-1$ is the size of $N(u_1)\setminus (N(u_2)\cup \{u_2\})$ and $\delta$ is the size of $N(u_2)\setminus (N(u_1)\cup \{u_1\})$. For any $1\le h \le m$, we define $A_{1,h},A_{2,h},\dots,A_{\delta,h}$ and put $A_{\ell,h}\cup B_h$ into $\cG$ for all $1\le \ell\le \delta$ as follows. Let $x_1,x_2,\dots,x_d$ be the elements of $D$, and $A_{1,h}$ be a $(v-4)$-element set containing $x_1$ and $e_\alpha$ (with $\alpha$ defined in the previous paragraph), and for $2\le \ell \le d$ let $A_{\ell,h}$ be an arbitrary $(v-4)$-element subset of $C$ containing $x_1,x_\ell$. (We need $v-4\ge 3$ to be able to make the choice of $A_{1,h}$.) Finally, let $A_{d+1,h},A_{d+2,h},\dots,A_{\delta,h}$ be distinct $(v-4)$-element subsets of $C\setminus \{x_1\}$. There are $v-3$ such subsets, each missing one element of $C\setminus\{x_1\}$. We take them one by one, starting with those that miss an element from $D\setminus \{x_1\}$.
%starting with those such subsets that miss another element of $D$ apart from $x_1$. 
The choice of $A_{1,h}$ verifies item 4 of the $F$-good property and items 1 and 2 hold by definition. 

To see item 3, let $1\le i<j\le m$. We need to create a copy of a Berge-$F'$. Vertices of $D$ will play the role of $N(u_1)\cap N(u_2)$, $A_{1,i}\cup B_i,A_{2,i}\cup B_i,\dots,A_{d,i}\cup B_i$ will play the role of the edges connecting $u_1$ to all the vertices of $D$ and similarly  $A_{1,j}\cup B_j,A_{2,j}\cup B_j,\dots,A_{d,j}\cup B_j$ will play the role of the edges connecting $u_2$ to all the vertices of $D$. To finish the Berge copy of $F'$ we will connect both $u_1$ and $u_2$ to all the vertices in $C\setminus D$ (thus in fact we present a Berge copy of $K_{2,v-2}$, which clearly contains $F'$). We will use the hyperedges $A_{i,d+1},A_{i,d+2},\dots,A_{i,d_1-1},A_{i,d+1}$ to connect $u_1$ to the vertices in $C\setminus D$. As they each contain $u_1$, it is enough to show an injection $f$ from $A_{i,d+1},A_{i,d+2},\dots,A_{i,d_1-1},A_{i,d+1},A_{j,d+1},A_{j,d+2},\dots,
\break A_{j,d_1-1},A_{j,d+1}$ to $C\setminus D$ such that $f(H) \in H$ for all sets.
%(and another injection $f'$ from $A_{j,d+1},A_{j,d+2},\dots,A_{j,d_1-1},A_{j,d+1}$ to $C\setminus D$ with $f'(H)\in H$ for all the sets).
%need to make sure that there exists a injection $f$ from $A_{i,d+1},A_{i,d+2},\dots,A_{i,d_1-1},A_{i,d+1},A_{j,d+2},\dots,A_{j,\delta}$ to $C\setminus D$ such that $f(H) \in H$ for all sets. If so, then the $f(A_{i,d+\ell})$'s play the role of neighbors of $u_1$ and the $f(A_{j,d+\ell})$'s play the role of neighbors of $u_2$. 
All we need to check is whether Hall's condition holds: as for any two distinct sets, their union contains $C\setminus D$, the only problem can occur if $A_{i,\ell}\cap (C\setminus D)=A_{j,\ell'}\cap (C\setminus D)\neq C\setminus D$ and $|C\setminus D|=1$ or 2. But then by $v-2\ge 5$, we have $d\ge 3$ and thus all choices of $A_{i,\ell},A_{j,\ell}$ contain $C\setminus D$ by the assumption that we picked those such subsets first that miss another element of $D$ apart from $x_1$.

Suppose next $d=0$. Then for any $1\le h \le m$ let us fix $\pi_h$, a permutation $z_1,z_2,\dots,z_{v-2}$ of vertices of $C$ with $z_1,z_2$ being the endvertices of the edge $e_\alpha$. Now let $A_{1,h},A_{2,h},\dots, A_{\delta,h}$ be cyclic intervals of length $v-4$ of $\pi_h$ with $e_\alpha \subset A_{1,h}$. Then putting the sets of the form $A_{\ell,h}\cup B_h$ to $\cG$ will satisfy items 1 and 2 by definition, item 4 by the choice of $A_{1,h}$, and item~3 by a similar Hall-condition reasoning as in the case of $d>0$.
\end{proof}

Now we are ready to prove the theorem. If $\delta=0$, then $F$ contains an isolated edge, and we are done by Proposition \ref{isoedge}. Otherwise by Claim \ref{Fgood} there exists an $F$-good hypergraph $\cG$ with $O(n)$ hyperedges, and by Claim \ref{extend} any $F$-saturated extension of $\cG$ has a linear number of hyperedges.
\end{proof}

\section{Concluding remarks}

For any graph $F$, integer $r\ge 2$ and enumeration $\pi: G_1,G_2,\dots,G_{\binom{n}{r}}$ of $\binom{[n]}{r}$ we can define a greedy algorithm that outputs a Berge-$F$-saturated $r$-uniform hypergraph $\cG$ as follows: we let $\cG_0=\emptyset$, and then for any $i=1,2,\dots,\binom{n}{r}$ we let $\cG_i=\cG_{i-1}\cup \{G_i\}$ if $\cG_{i-1}\cup \{G_i\}$ is Berge-$F$-free, and $\cG_i=\cG_{i-1}$ otherwise. Clearly, $\cG_{\pi,r}=\cG_{\binom{n}{r}}$ is Berge-$F$-saturated.

\begin{proposition}\label{colex}
If for some graph\/ $F$ we have\/ $\delta(F)=\kappa(F)$, then the greedy algorithm with respect to the colex order outputs an\/ $F$-saturated graph\/ $G$ with a linear number of edges.
\end{proposition}

\begin{proof}
Observe that no matter what $F$ is, the greedy algorithm starts by creating a clique on $[v(F)-1]$ and later every vertex $u \in [n]\setminus [v(F)-1]$ will be connected to all vertices $u' \in [\delta(F)-1]$. Indeed, a new vertex cannot help creating a copy of $F$ without having degree at least $\delta(F)$. We claim that the graph $G$ will consist of $\delta(F)-1$ universal vertices, $\lfloor \frac{n-\delta(F)+1}{v(F)-\delta(F)}\rfloor$ cliques of size $v(F)-\delta(F)$, and one possible extra clique of size $s$, where $s$ is the residue of $n-\delta(F)+1$ modulo $v(F)-\delta(F)$. The universal vertices with any clique of size $v(F)-\delta(F)$ form a clique of size $v(F)-1$, so any new vertex is joined to the universal vertices, and the next $v(F)-\delta(F)$ vertices can form a clique because of the condition $\delta(F)=\kappa(F)$.
\end{proof}

Observe that the Erd\H os-R\'enyi random graph $G(k,1/2)$ satisfies the condition of Proposition \ref{colex} with probability tending to 1 as $k$ tends to infinity. So, for almost all graphs $F$, the greedy algorithm outputs  an $F$-saturated graph $G$ with a linear number of edges.

\begin{proposition}
Suppose\/ $F$ is connected and contains a cut-edge\/ $u,w$. Then\/ $r\ge e(F)$ implies\/ $sat_r(n,F)=O(n)$.
\end{proposition}

\begin{proof}
It is known that the saturation number of stars is linear, so we can assume that $F$ is not a star. Let $F_u$ denote the component of $u$ in $F\setminus \{(u,w)\}$ and $F_w$ denote the component of $w$ in $F\setminus \{(u,w)\}$. Let us consider a partition $B_0,B_1,B_2,\dots, B_m$ of an $n$-element set $U$ with $|B_1|=|B_2|=\dots =|B_m|=r+1$ and $|B_0|\le r$. For $i=1,2,\dots,m$, let $\cG_{0,i}$ consist of $e(F)-1$ $r$-subsets of $B_i$ and let $\cG_0=\cup_{i=1}^m\cG_{0,i}$. Clearly, $\cG_0$ is Berge-$F$-free as its components contain $e(F)-1$ hyperedges. We claim that if $G$ contains vertices $u_i \in B_i$ and $u_j\in B_j$ with $1\le i< j\le m$, then $\{G\}\cup \cG_0$ contains a Berge-$F$. Indeed, $G$ can play the role of the cut-edge with $u_i$ and $u_j$ as its two endpoints. We need to show that $\cG_{0,i}$ contains a Berge-$F_u$ with $u$ played by $u_i$. (The proof for $\cG_{0,i}$ containing a Berge-$F_w$ with $w$ played by $u_j$ is identical.) The graph $F_u$ contains at most $e(F)-1$ edges. We need to verify Hall's condition in the auxiliary bipartite graph $B$ with one part the edges of $F_u$ and the other part the hyperedges of $\cG_{0,i}$ and an edge $e$ is connected to a hyperedge $G$ if and only if $e\subset G$. Note that the degree of any edge $e$ is at least $e(F_u)-2$ and if $e_1$ and $e_2$ are disjoint edges of $F_u$, then in $B$ their neighborhood is $\cG_{0,i}$. Therefore the only problem that can occur is if $F_u$ is a star with $e(F)-1$ leaves. If the center of $F_u$ is $u$, then $F$ is also a star, contradicting our assumption. If the center $c$ of $F_u$ is not $u$, then a vertex $u'\in B_i$ that is contained in all hyperedges of $\cG_{0,i}$ can play the role of $c$.
\end{proof}

\section*{Acknowledgement}

Research was supported by the National Research, Development and Innovation Office - NKFIH under the grants FK 132060, KH130371, KKP-133819 and SNN 129364. Research of  Vizer was supported by the J\'anos Bolyai Research Fellowship. and by the New National Excellence Program under the grant number \'UNKP-20-5-BME-45.
Research of Patk\'os was supported by the Ministry of Educational and Science of the Russian Federation in the framework of MegaGrant no. 075-15-2019-1926.

\end{document}